\documentclass[12pt]{amsart}

\usepackage{amsmath}
\usepackage{amscd}
\usepackage{amssymb}
\usepackage{latexsym}
\usepackage{stmaryrd}
\usepackage{mathbbol}
\usepackage{undertilde}
\usepackage{mathabx}
\usepackage{color}

\input xy
\xyoption{all} \CompileMatrices

\definecolor{orange}{rgb}{1,0,0}
\definecolor{cadmiumgreen}{rgb}{0, 0.78, 0.05}

\newtheorem{theorem}{Theorem}[section]
\newtheorem*{theorem*}{Theorem}
\newtheorem{lemma}[theorem]{Lemma}
\newtheorem*{lemma*}{Lemma}
\newtheorem{corollary}[theorem]{Corollary}
\newtheorem{proposition}[theorem]{Proposition}
\newtheorem{defprop}[theorem]{Proposition-Definition}
\newtheorem{remark}[theorem]{Remark}
\newtheorem{definition}[theorem]{Definition}

\makeatletter
\def\revddots{\mathinner{\mkern1mu\raise\p@
\vbox{\kern7\p@\hbox{.}}\mkern2mu
\raise4\p@\hbox{.}\mkern2mu\raise7\p@\hbox{.}\mkern1mu}}
\makeatother 
\newcommand{\bgl}{\begin{equation}} 
\newcommand{\egl}{\end{equation}}
\newcommand{\bgloz}{\begin{equation*}} 
\newcommand{\egloz}{\end{equation*}}
\newcommand{\bgln}{\begin{eqnarray}} 
\newcommand{\egln}{\end{eqnarray}}
\newcommand{\bglnoz}{\begin{eqnarray*}} 
\newcommand{\eglnoz}{\end{eqnarray*}}
\newcommand{\btheo}{\begin{theorem}}
\newcommand{\etheo}{\end{theorem}}
\newcommand{\btheooz}{\begin{theorem*}}
\newcommand{\etheooz}{\end{theorem*}}
\newcommand{\blemma}{\begin{lemma}}
\newcommand{\elemma}{\end{lemma}}
\newcommand{\blemmaoz}{\begin{lemma*}}
\newcommand{\elemmaoz}{\end{lemma*}}
\newcommand{\bproof}{\begin{proof}}
\newcommand{\eproof}{\end{proof}}
\newcommand{\bbew}{\begin{beweis}}
\newcommand{\ebew}{\end{beweis}}
\newcommand{\bremark}{\begin{remark}\em}
\newcommand{\eremark}{\end{remark}}
\newcommand{\bdefin}{\begin{definition}}
\newcommand{\edefin}{\end{definition}}
\newcommand{\bprop}{\begin{proposition}}
\newcommand{\eprop}{\end{proposition}}
\newcommand{\bdefprop}{\begin{defprop}}
\newcommand{\edefprop}{\end{defprop}}
\newcommand{\bcor}{\begin{corollary}}
\newcommand{\ecor}{\end{corollary}}
\newcommand{\bfa}{\begin{cases}} 
\newcommand{\efa}{\end{cases}}
\newcommand{\nn}{\par\vspace{-3mm}\noindent}

\newcommand{\cB}{\mathcal B}

\newcommand{\cE}{\mathcal E}

\newcommand{\cI}{\mathcal I}

\newcommand{\cK}{\mathcal K}

\newcommand{\cM}{\mathcal M}

\newcommand{\cO}{\mathcal O}

\def\Cz{\mathbb{C}}

\def\Nz{\mathbb{N}}

\def\Zz{\mathbb{Z}}

\def\1z{\mathbb{1}}
\newcommand{\lori}{\longrightarrow}

\newcommand{\ve}{\varepsilon}
\newcommand{\vp}{\varphi}

\def\SEMI{\mbox{$\times\kern-2pt\vrule height5pt width.6pt \kern3pt $}}

\newcommand{\id}{{\rm id}}
\newcommand{\alg}{{\rm alg}}

\newcommand{\Ad}{{\rm Ad\,}}

\newcommand{\nuc}{\mathrm{nuc}}
\begin{document}
\title[$KK$ with extra structures]{Generalized homomorphisms and $KK$ with extra structures}

\author{Joachim Cuntz}
\address{Mathematisches Institut\\
  Universit\"at M\"unster\\
  Ein\-stein\-str.\ 62\\
  48149 M\"unster\\
  Germany}
\email{cuntz@uni-muenster.de}

\author{James Gabe}
\address{Department of Mathematics and Computer Science\\
  University of Southern Denmark\\
  Campusvej 55\\
  DK-5230 Odense\\
  Denmark}
\email{gabe@imada.sdu.dk}


\begin{abstract}
  We develop the approach via quasihomomorphisms and the universal algebra $qA$ to Kasparov's $KK$-theory, so as to cover versions of $KK$ such as $KK^{nuc}$, $KK^G$ and ideal related $KK$-theory.
\end{abstract}
\thanks{2010 Mathematics Subject Classification. Primary 19K35, 46L80; Secondary 47L35, 19L47.}
\thanks{The first named author was supported by Deutsche Forschungsgemeinschaft (DFG) via Exzellenzstrategie des Bundes und der L\"ander EXC 2044 –390685587, Mathematik M\"unster: Dynamik–Geometrie–Struktur.\\ The second named
  author was supported by the IRFD grants 1054-00094B and
1026-00371B}

\maketitle
\section{Introduction}

Kasparov's $KK$-theory is a main tool in the theory of operator algebras and noncommutative geometry. It is based on a very flexible but not easy formalism developed by Kasparov. In \cite{CuGen} and \cite{CuKK} the first named author has introduced an alternative more algebraic approach based on quasihomomorphisms and the universal algebra $qA$ associated with an algebra $A$. In this picture elements of $KK(A,B)$ are represented by homomorphisms from $qA$ to $\cK\otimes B$ where $\cK$ denotes the standard algebra of compact operators on $\ell^2\Nz$. One merit of this approach is a simple and universal construction of the product in $KK$ from which in particular  associativity becomes very natural. Since many important $KK$-elements come naturally from quasihomomorphisms, at the same time it can be used to treat $KK$-elements that occur in `nature'. Note that there are possible definitions of $KK(A,B)$ that make the product and its associativity automatic but have the disadvantage that $KK$-elements appearing in applications never fit the definition naturally - take for instance the possible definition as homotopy classes of homomorphisms from $\cK\otimes qA$ to $\cK\otimes qB$. There also is the approach of \cite{CuLoc}, \cite{CuTh} which is based on the use of the universal algebra $qA$ too, and works also for Banach and locally convex algebras and in fact even much more general algebras \cite{CoTh},\cite{CMR}. The definition and especially the product however uses higher quasihomomorphisms (maps from $q^nA$ rather than from $qA$). In applications to C*-algebras e.g. for classification this is not good enough because there it is usually important that a $KK$-element can be represented by a prequasihomomorphism instead of a Kasparov-module.

One strength of Kasparov's formalism is the fact that by now it has been extended to define very useful versions of $KK$ for categories of C*-algebras with additional structure such as equivariant $KK$-theory \cite{KasInv}, $KK^{nuc}$ \cite{SkandKKnuc} or ideal related $KK$-theory \cite{Kir00}. In this article we adapt the formalism of \cite{CuKK} to allow for these additional structures. We will give definitions of the various $KK$-theories using the approach via the universal algebra $qA$ and establish the associative product in each case. In section \ref{suni} we will explain that our construction reproduces the $KK$-theories defined previously in the papers cited above. Moreover we will see there that in the case of equivariant and ideal related $KK$-theory we obtain a universal functor with the usual properties of split exactness, homotopy invariance and stability.\\
An nice feature of our approach is the fact that the ideal preserving or nuclearity condition on a homomorphism $\vp:qA\to B$ can be characterized by a simple criterion. In fact, these conditions can already be checked on the linear map $A\ni x\mapsto \vp(qx)$ (where $qx$ is one of the standard generators of $qA$).
This description of $KK^{nuc}$ will be used in upcoming work of the second named author \cite{CGSTW2} to simplify functoriality of this functor similar to how this formalism was used in \cite[Appendix B.1]{CGSTW1}.

The most established and probably the most important of the $KK$-theories we discuss is the equivariant theory $KK^G$. This version of $KK$ has been discussed on the basis of the $qA$ approach by Ralf Meyer in \cite{MeyEqui}. In fact one basic idea in his approach appears also in our discussion. We mention however that Meyer does not touch the Kasparov product at all. Using Meyer's result we get a new description of the product in Kasparov's $KK^G$.
\\
For the construction of the product we will not use Kasparov's technical theorem as in \cite{KasInv} or Pedersen's derivation lifting theorem as in \cite{CuKK} but Thomsen's somewhat simpler noncommutative Tietze extension theorem \cite[1.1.26]{JeTh}. In the equivariant case we will also need a new equivariant version of this theorem which we prove in section \ref{s0}.

\section{Preliminaries}\label{s0}
Notation: In the following, homomorphisms between C*-algebras will always be assumed to be  *-homomorphisms. By $\cK$ we denote the standard algebra of compact operators on $\ell^2 \Nz$. There is a natural isomorphism $\cK\cong \cK\otimes \cK$. A C*-algebra $A$ is called stable if $A\cong \cK\otimes A$. Given a C*-algebra $A$ we denote by $\cM (A)$ its multiplier algebra. If $\vp:A\to B$ is a $\sigma$-unital homomorphism between C*-algebras, we denote by $\vp^\circ$ its extension to a homomorphism $\cM (A)\to \cM (B)$.

Let $A$ be a C*-algebra. We denote by $QA$ the free product $A\star A$ and by $\iota, \bar{\iota}$ the two natural inclusions of $A$ into $QA=A\star A$. We denote by $qA$ the kernel of the natural map $A\star A\to A$ that identifies the two copies $\iota (A)$ and $\bar{\iota} (A)$ of $A$. Then $qA$ is the closed two-sided ideal in $QA$ that is generated by the elements $qx=\iota (x) - \bar{\iota} (x),\,x\in A$.\\
There is the natural evaluation map $\pi_A: qA \to A$ given by the restriction to $qA$ of the map $\id\star 0:QA\to A$ that is the identity on the first copy of $A$ and zero on the second one.
\bprop\label{gen}
For $x,y\in A$  one has the identity $$q(xy)=\iota (x)q(y) + q(x)\bar{\iota} (y)=\bar{\iota} (x)q(y) + q(x)\iota (y)$$
Finite sums of elements of the form $\iota (x_0)qx_1\ldots qx_n$ and $qx_1, \ldots qx_n$ or of the form $qx_1\ldots qx_n\iota (x_0)$ and $qx_1, \ldots qx_n$ are dense in $qA$. In particular $qA$ is generated as a closed left or right ideal in $qA$ by the elements $qx$, $x\in A$.  \eprop
\bproof
The identity for $q(xy)$ is trivially checked. The other statements are consequences (for the assertion on the generation as a closed left or right ideal note that $\iota (y)qx$ is the limit of $\iota (y)u_\lambda qx$ for an approximate unit $(u_\lambda )$ in $qA$).\eproof
As in \cite{CuKK} we define a prequasihomomorphism between two C*-algebras $A$ and $B$ to be a diagram of the form $$A \quad \stackrel{\vp,\bar{\vp}}{\rightrightarrows}\quad \cE \rhd J \stackrel{\mu}{\to} B$$
i.e. two homomorphisms $\vp,\bar{\vp}$ from $A$ to a C*-algebra $\cE$ that contains an ideal $J$, with the condition that $\vp (x)-\bar{\vp} (x)\in J$ for all $x\in A$ and finally a homomorphism $\mu : J\to B$. The pair $(\vp,\bar{\vp})$ induces a homomorphism $QA\to \cE$ by mapping the two copies of $A$ via $\vp,\bar{\vp}$. This homomorphism maps the ideal $qA$ to the ideal $J$. Thus, after composing with $\mu$, every such prequasihomomorphism from $A$ to $B$ induces naturally a homomorphism $q(\vp,\bar{\vp}):qA\to B$. Conversely, if $\psi: qA \to B$ is a homomorphism, then we get a prequasihomomorphism by choosing $\cE =\cM(\psi (qA)), \,J=\psi(qA)$ and $\vp=\psi^\circ \iota,\,\bar{\vp}=\psi^\circ \bar{\iota}$ as well as the inclusion $\mu: \psi(qA) \hookrightarrow B$.

In this paper we will also have to use an iteration of the $qA$ construction. We will write $Q^2A$ for the free product $Q(QA)=QA\star QA$ and $\eta,\bar{\eta}$ for the two natural embeddings of $QA$ into $Q^2A$. We now denote by $\ve, \bar{\ve}$ the two embeddings $A\to QA$ and get four embeddings $\eta\ve, \eta\bar{\ve}, \bar{\eta}\ve, \bar{\eta}\bar{\ve}$ of $A$ to $Q^2A$. We have the ideal $qA$ generated by the elements $\ve (x)-\bar{\ve} (x), \, x\in A$ in $QA$ and the ideal $q^2A$ generated by $\eta (z)- \bar{\eta} (z)$, $z\in qA$ in $Q(qA)$.

In Section \ref{sequi} we will use the following equivariant version of Thomsen's noncommutative Tietze extension theorem which we prove here. Recall that when $G$ is a locally compact group, a $G$-$C^\ast$-algebra $A$ is a $C^\ast$-algebra with a point-norm continuous action $\alpha$ of $G$ on $A$. This action extends to a point-strictly continuous action $\alpha^\circ$ on the multiplier algebra $\mathcal M(A)$, where we remark that each automorphism $\alpha_g^\circ$ for $g\in G$ is strictly continuous on bounded sets. To simplify notation, we will sometimes write $g\cdot a$ instead of $\alpha_g(a)$ for $a\in A$ and $g\in G$ (or instead of $\alpha^\circ_g(a)$ if $a\in \mathcal M(A)$).

\bprop\label{eqT} Let $G$ be a locally compact $\sigma$-compact group, let $0\to J\to A\stackrel{\pi}{\to} B\to 0$
be an extension of $\sigma$-unital $G$-C*-algebras, and let $X\subset \cM (A)$
be a norm-separable self-adjoint subspace. Let $\pi^\circ : \cM(A) \to \cM(B)$ be the
induced homomorphism. For every $z$ in the commutator $\cM (B)\cap \pi^\circ (X)'$ of $\pi^\circ (X)$ in $\cM (B)$, such that $g\cdot z=z$
for all $g \in G$ there exists $y \in \cM(A)$ such that $\pi^\circ (y)=z,\, [y,X]\subseteq J,\,g\cdot y-y\in J$ for all $g\in G$ and $G\ni g \mapsto g\cdot y$ is norm-continuous.
\eprop
\bproof
We may assume without loss of generality that $z$ is a positive contraction. Let $h\in A$ be strictly positive, let $\mathcal F \subset X$ be a compact subset of contractions with dense span,\footnote{If $(x_n)_{n\in \mathbb N}$ is a dense sequence in the unit ball of $X$ one could pick $\mathcal F = \{ \tfrac{1}{n} x_n : n\in \mathbb N\}\cup \{0\}$.} and let $H_1 \subseteq H_2 \subseteq \dots \subseteq G$ be compact neighbourhoods of the identity such that $G = \bigcup H_n$. Since $B$ is also $\sigma$-unital, we apply \cite[Lemma 1.4]{KasInv} and pick a (positive, increasing, contractive) approximate identity $(e_n)_{n\in \mathbb N}$ for $B$ such that
\begin{eqnarray}
\| (1 - e_n) z^{1/2} \pi(h) \| & \leq & 4^{-n} \label{eq:i3} \\
\sup_{x\in \mathcal F} \| \pi^\circ (x)  e_n  -  e_n  \pi^\circ (x) \| & \leq & 4^{-n} \label{eq:i1} \\
\sup_{g\in H_n} \| g \cdot e_n - e_n \| & \leq & 4^{-n} \label{eq:i4}
\end{eqnarray}
for $n\in \mathbb N$. To ease notation let $e_0 = 0$.  We will recursively construct positive contractions $0 = y_0 \leq y_1 \leq y_2 \leq \dots $ in $A$ such that for $n\in \mathbb N$
\begin{eqnarray}
\pi(y_n) & = & z^{1/2} e_n z^{1/2} \label{eq:1}  \\
\| (y_{n+1} - y_n) h\| &\leq & 2^{-n}  \label{eq:2} \\
\sup_{x\in \mathcal F} \| [y_{n+1}- y_n, x]\| & \leq & 2^{-n}  \label{eq:3} \\
\sup_{g\in H_n} \| g\cdot (y_{n+1} - y_n) - (y_{n+1} - y_n)\| & \leq & 2^{-n}. \label{eq:4}
\end{eqnarray}
Letting $y_0 = 0$, suppose we have constructed $y_0\leq \dots \leq y_n$ as above. We will explain how to construct $y_{n+1}$.

Since  $z^{1/2} (e_{n+1} - e_n) z^{1/2} \leq 1 - z^{1/2} e_n z^{1/2}$, we apply \cite[Proposition 1.5.10]{Ped} to pick $c \in A$ such that $\pi(c) = z^{1/2} (e_{n+1} - e_n) z^{1/2}$ and $0\leq c \leq 1-y_n$ in $\tilde A$.  Again using \cite[Lemma 1.4]{KasInv} we let $(v_k)_{k\in \mathbb N}$ be an approximate identity in $J$ which is quasi-central relative to $\{c, y_n, h \} \cup \mathcal F$ and such that $\lim_{k\to \infty} \sup_{g\in H_n} \| g \cdot v_k - v_k \| =0$.  Let $y_{n+1}^{(k)} := y_n +  c^{1/2} (1- v_k) c^{1/2}$. We will show that we can pick $y_{n+1} = y_{n+1}^{(k)}$ for sufficiently large $k$.

That \eqref{eq:1}, \eqref{eq:2}, and \eqref{eq:3} are satisfied is exactly as in the proof of \cite{JeTh}, so it remains to show \eqref{eq:4}.
For this we compute
\begin{eqnarray*}
&& \limsup_{k\to \infty}\sup_{g\in H_{n}} \| g\cdot (y_{n+1}^{(k)}- y_n) - (y_{n+1}^{(k)}-y_n)\| \\
&=&  \limsup_{k\to \infty} \sup_{g\in H_{n}} \| g\cdot ((1-v_k) c) - (1-v_k) c\|  \\
&=&  \limsup_{k\to \infty} \sup_{g\in H_{n}} \| (1-v_k) (g\cdot c - c)\| \\
&=& \sup_{g\in H_{n}} \| g\cdot (z^{1/2} (e_{n+1} - e_n) z^{1/2}) - z^{1/2} (e_{n+1} - e_n) z^{1/2}\|  \\
&=& \sup_{g\in H_{n}} \| z^{1/2} (g\cdot (e_{n+1} - e_n) - (e_{n+1} - e_n) ) z^{1/2}\| \\
&\stackrel{\eqref{eq:i4}}{\leq }&  2^{-n}.
\end{eqnarray*}
Hence we may define $y_{n+1}  = y_{n+1}^{(k)}$ for large $k$ so that it satisfies \eqref{eq:1}--\eqref{eq:4}, so we obtain our desired sequence $(y_m)_{m\in \mathbb N}$.

By \eqref{eq:2} it follows that $(y_n)_n$ converges strictly to a positive contraction $y\in \mathcal M(A)$. Since $\pi^\circ$ is strictly continuous on bounded sets, it follows from \eqref{eq:1} that $\pi^\circ (y) = z$ (since $z$ is the strict limit of $z^{1/2} e_n z^{1/2}$). For $x\in \mathcal F$ we have by \eqref{eq:3} that $[y_n, x]$ norm-converges to an element in $A$, so that $[y,x] \in A$. Moreover,
\[
\pi([y,x]) = \lim_{n \to \infty} \pi^\circ ([y_n, x]) \stackrel{\eqref{eq:1}}{=} \lim_{n\to \infty} z^{1/2} [e_n, \pi^\circ(x)] z^{1/2} \stackrel{\eqref{eq:i1}}{=} 0
\]
so that $[y,x] \in J$ for all $x\in \mathcal F$. Hence $[y,x] \in J$ for all $x\in \overline{\mathrm{span}} \mathcal F = X$.

As the $G$-action on $\mathcal M(A)$ is pointwise strictly continuous, it follows that $g\cdot y$ is the strict limit of $(g\cdot y_n)_{n\in \mathbb N}$ for any $g\in G$. By \eqref{eq:4}, $(g\cdot y_n - y_n)_{n\in \mathbb N}$ converges in $A$ as $n\to \infty$ for every $g\in G$. Hence $g\cdot y-y\in A$. Moreover,
\begin{eqnarray*}
\pi(g\cdot y-y) &=& \lim_{n\to \infty} \pi^\circ(g\cdot y_n - y_n) \\
&\stackrel{\eqref{eq:1}}{=}& \lim_{n\to \infty} g\cdot (z^{1/2}e_n z^{1/2}) - z^{1/2} e_n z^{1/2} \\
&=& \lim_{n\to \infty} z^{1/2} (g\cdot e_n - e_n)z^{1/2} \\
&\stackrel{\eqref{eq:i4}}{=}&  0.
\end{eqnarray*}
Hence $g\cdot y-y\in J$ for all $g\in G$.

Finally, given $\epsilon>0$, pick $N \in \mathbb N$ such that $\sum_{k=N}^\infty 2^{-n} < \epsilon$. Choose an open neighbourhood $U \subseteq H_N \subseteq G$ of the identity such that $\sup_{g\in U} \| g\cdot y_N - y_N\| < \epsilon$. Then
\begin{eqnarray*}
\sup_{g\in U} \| g\cdot y - y \| &=& \sup_{g\in U} \| \sum_{k=N}^\infty (g\cdot (y_{k+1} - y_k)  - (y_{k+1} - y_k)) + g\cdot y_N - y_N \| \\
&\stackrel{\eqref{eq:4}}{\leq}& \epsilon + \sup_{g\in U}\| g\cdot y_N - y_N \| \\
&<& 2\epsilon.
\end{eqnarray*}
Hence $G\ni g \mapsto g\cdot y\in \mathcal M(A)$ is norm-continuous.
\eproof
\section{The product in $KK$}\label{s3}
Given two homomorphisms $\vp,\psi:X\to Y$ between C*-algebras we denote by $\vp\oplus\psi$ the homomorphism $$x\mapsto\;\scriptsize{\left(\begin{matrix}
 \vp (x) & 0\\
 0 & \psi (x)
\end{matrix}\right)}$$ from $X$ to $M_2(Y)$. Following \cite{CuKK} we define
\bdefin
Let $A$, $B$ be C*-algebras and $qA$ as in Section \ref{s0}. We define $KK(A,B)$ as the set of homotopy classes of homomorphisms from $qA$ to $\cK\otimes B$.
\edefin
The set $KK(A,B)$ becomes an abelian group with the operation $\oplus$ that assigns to two homotopy classes $[\vp],[\psi]$ of homomorphisms $\vp,\psi:qA\to \cK\otimes B$ the homotopy class $[\vp\oplus \psi ]$ (using an isomorphism $M_2(\cK) \cong \cK$ to identify $M_2(\cK\otimes B)\cong \cK\otimes B$, which is well-defined since such an isomorphism is unique up to homotopy). In \cite{CuGen} it was checked that this definition of $KK(A,B)$ is equivalent to the one by Kasparov. We recapitulate now the construction in \cite{CuKK} of the product $KK(A,B)\times KK(B,C)\to KK (A,C)$. It is based on a functorial map $\vp_A: qA\to M_2(q^2A)$ (which is in fact - up to stabilization by the $2\times 2$- matrices $M_2$ - a homotopy equivalence). Since versions of this map and of its properties will be used in each of the subsequent sections on $KK$ with additional structure we include complete proofs. We take this opportunity to include more details on the proofs and to arrange the arguments given in \cite{CuKK} in a slightly different way.\\
To prove the existence of the map $\vp_A$ we will use Proposition \ref{eqT} with $A$ in place of $X$. Since $X$ in \ref{eqT} has to be separable we will assume in this section and in later sections where we discuss the product of $KK(A,B)$ and $KK(B,C)$ to $KK(A,C)$ with extra structure that $A$ is separable.

Given a C*-algebra $A$, we use the four embeddings $\eta\ve, \eta\bar{\ve}, \bar{\eta}\ve, \bar{\eta}\bar{\ve}$ of $A$ to $Q^2A$ from section \ref{s0}. Consider the C*-algebra $R$ generated by the matrices $$\left(\begin{matrix}
 R_1 & R_1R_2\\
 R_2R_1 & R_2
\end{matrix}\right)$$
where $R_1=\eta(qA)$, $R_2=\bar{\eta}(qA)$. Consider also the C*-algebra $D$ generated by matrices of the form
$$D = \left(\begin{matrix}
 \eta\ve (x) & 0\\
 0 & \bar{\eta}\ve (x)
\end{matrix}\right)\quad x\in A$$
Then $R$ is a subalgebra of $M_2(QqA)$ where $QqA$ is the C*-subalgebra of $Q^2A$ generated by $\eta (qA)$ and $\bar{\eta} (qA)$. Let $J= R\cap M_2(q^2A)$. Since $q^2 A$ is an ideal in $QqA$ this is an ideal in $R$. One also clearly has $DR, RD\subset R$. Thus $R$ is an ideal in $R+\!D$ and $J$ is also an ideal of $R+\!D$ (we think of all these algebras as subalgebras of $M_2(Q^2 A)$).

Because $\eta (qA)/q^2 A = \bar{\eta} (qA)/q^2 A\cong qA$, the quotient $R/J$ is isomorphic to $M_2(qA)$. Moreover $(R+\!D)/J$ is isomorphic to the subalgebra of $M_2(Q(A))$ generated by $M_2(qA)$ together with the matrices
$$\left(\begin{matrix}
 \iota (x) & 0\\
 0 & \iota (x)
\end{matrix}\right) \quad x\in A$$

If $A$ is separable we can use Thomsen's noncommutative Tietze extension theorem \cite[1.1.26]{JeTh} (see also Proposition \ref{eqT}) and lift the multiplier
$$
\left(\begin{matrix}
 0 & 1\\
 1 & 0
\end{matrix}\right)
$$
of $R/J$
to a self-adjoint multiplier $S$ of $R$ that commutes mod $J$ with $D$.

We can now set
$F=e^{\frac{\pi i}{2} S}$ and define the automorphism $\sigma$ of $\cM (J)$ by $\Ad F$.

Consider the homomorphisms $A\to \cM (J)$ given by
$$
h_1= \left(\begin{matrix}
 \eta\ve & 0\\
 0 & \bar{\eta}\bar{\ve}
\end{matrix}\right)
,\qquad
h_2 =\left(\begin{matrix}
 \eta\bar{\ve} & 0\\
 0 & \bar{\eta}\ve
\end{matrix}
\right)
$$
In the following we use the notation $\oplus$ introduced at the beginning of the section. Thus $h_1=\eta\ve\oplus \bar{\eta}\bar{\ve}$ and $h_2=\eta\bar{\ve}\oplus\bar{\eta}\ve$.\\
\bdefin We define the homomorphism $\vp_A : qA\to J \subset M_2(q^2A)$ by the prequasihomomorphism given by the pair of homomorphisms  $(h_1,\sigma h_2)$   (compare \cite{CuKK}, p.39), i.e. $\vp_A=q(h_1,\sigma h_2)$.\edefin
To check that the difference of $h_1$ and $\sigma h_2$ maps to $J$ recall that by definition $\sigma$ fixes $d(x)=\eta\ve (x)\oplus\bar{\eta}\ve (x)$ mod $J$ for each $x\in A$ and that $h_2(x)= d(x) - \eta (q(x))\oplus 0$. The term $\eta q(x)\oplus 0$ is moved by $\sigma$ to $0\oplus\bar{\eta} q(x)$ mod $J$ (note that $\eta q(x)-\bar{\eta} q(x)\in q^2A$). Since $\bar{\eta}\ve(x)- \bar{\eta} (qx)=\bar{\eta}\bar{\ve} (x)$ we get that $\sigma h_2(x)=h_1(x)$ mod $J$.

Note the $\vp_A$ is unique up to homotopy. In fact, if we picked a different operator $S_1 \in \mathcal M(R)$ instead of $S$ as above, and define $S_t = (1-t)S + tS_1$ and $\sigma_t = \Ad e^{\tfrac{\pi i}{2} S_t}$, then $q(h_1, \sigma_t h_2)$ defines a homotopy from $q(h_1, \sigma h_2)$ to  $q(h_1, \sigma_1 h_2)$.

\subsection{The Kasparov product via the universal map $\vp_A$}\label{sub prod}
Once the map $\vp_A$ is constructed we can define the product  $KK(A,B)\times KK(B,C)\to KK (A,C)$ as follows. \\
Let $\alpha :qA\to \cK\otimes B$ and  $\beta :qB\to \cK\otimes C$ represent elements $a\in KK(A,B)$ and $b\in KK(B,C)$ respectively. Since $q$ is a functor, we can form the homomorphism $q(\alpha):q^2 A \to q(\cK\otimes B)$. The pair of homomorphisms $(\id_\cK \otimes \iota, \id_\cK \otimes \bar{\iota})$ gives a natural map $\mu:q(\cK\otimes B)\to \cK\otimes qB$. The product of $a$ and $b$ is then represented by the following composition
\bgl\label{forprod} qA\, \stackrel{\vp_A}{\to}\, q^2 A \stackrel{q(\alpha)}{\to} q(\cK\otimes B)
\stackrel{\mu}{\to} \cK\otimes qB\, \stackrel{\id_\cK\otimes \beta}{\to}\, \cK\otimes\cK\otimes C\cong \cK\otimes C\egl
For simplicity we have left out the tensor product by the $2\times 2$-matrices $M_2$ which can be absorbed in the tensor product by $\cK$. Here and later we sometimes extend homomorphisms, such as $q(\alpha )$ here, tacitly to matrices or stabilizations. We denote the resulting homomorphism $qA\to \cK\otimes C$ in \eqref{forprod} by $\beta\,\sharp\, \alpha$.\\
This description of the product will be used in the subsequent sections in different versions.
\bremark\label{remrest}
(a) If $\alpha$ maps $qA$ to $B\subset \cK\otimes B$ then we can omit the map $\mu$ and the stabilization of $\beta$. We get that $\beta\,\sharp\,\alpha$ then is represented by $\beta q(\alpha)\vp_A$. The same formula applies if $\alpha$ maps $qA$ to $\cK\otimes B$ and $B\cong \cK\otimes B$.\\
(b) Assume that $B$ and $C$ are stable and let $\alpha :qA\to B$ and $\beta :qB\to C$ represent elements of $KK(A,B)$ and $KK(B,C)$. Denote by $\underline{B}$ the C*-subalgebra of $B$ generated by $\alpha (qA)$ and by $j_B$ the inclusion $\underline{B} \hookrightarrow B$. Let $\underline{\alpha} :qA\to \underline{B}$ and $\underline{\beta} = \beta\circ q(j_B) :q\underline{B} \to C$ denote the corestriction and restriction of $\alpha$ and $\beta$. Then we have $\underline{\beta}\,\sharp\, \underline{\alpha} = \beta\sharp \alpha$. In fact $\beta q(\alpha)\vp_A$ factors as $\beta\circ q(j_B) q(\underline{\alpha})\vp_A$ and the second expression represents $\underline{\beta}\,\sharp\, \underline{\alpha}$.\\
Instead of $\underline{B}$ we can just as well consider the hereditary subalgebra $B_0$ of $B$ generated by $\underline{B}$ and define $\alpha_0,\beta_0$ in analogy to $\underline{\alpha},\underline{\beta}$. We get the formula $\beta_0\,\sharp\, \alpha_0 = \beta\sharp \alpha$. We will use this setting below.
\eremark

\subsection{Associativity}\label{ass} The important point that gives associativity of the product is the existence of a homotopy inverse (up to tensoring by $M_2$) for $\vp_A$. It is given by  $\pi_{qA}:q^2 A\to qA$. We define $\pi_{qA}:QqA\to qA$ as the homomorphism that annihilates $\bar{\eta}(qA)$ in the free product $QqA=\eta qA\star \bar{\eta} qA$, and also as in Section \ref{s0} its restriction to $q^2A \subset QqA$.
\bprop\label{prop pi}
There is a continuous family of homomorphisms $\psi_t : q^2A\to M_2(q^2A), \,t\in [0,1]$ such that $\psi_0=\id_{q^2 A} \oplus 0$ and $\psi_1=\vp_A\pi_{qA} $.\\
There also is a continuous family of homomorphisms $\lambda_t: qA\to R\subset M_2(QqA)$ such that $\pi_{qA}\lambda_0 = \id_{qA}\oplus 0$ and $\pi_{qA}\lambda_1 = \pi_{qA}\vp_A$ (here and later we extend $\pi_{qA}: q^2A\to qA$ tacitly to a homomorphism $M_2 (q^2 A)\to M_2 (qA)$ between $2\times 2$-matrices).
\eprop
\bproof
Let $S$ be as above a lift of the multiplier given  on $R/J$ by the matrix
$$M=\left(\scriptsize{\begin{matrix}
 0 & 1\\
 1 & 0
\end{matrix}}\right)$$ to a multiplier of $R$ and denote by $S'$  the multiplier of $M_2(QqA)$ given by the same matrix $M$. For each $t\in [0,1]$ we let $\sigma_t$ denote the automorphism of $R$ given by $\Ad e^{\frac{\pi i}{2} St}$ and $\tau_t$ the automorphism of $M_2(QqA)$ given by $\Ad e^{\frac{\pi i}{2} S't}$.\\
Since $\sigma_t$ fixes the algebra $D$ from above pointwise mod $J$, the homomorphisms $ \eta \ve\oplus \bar{\eta} \bar{\ve} $ and $\sigma_t( \eta \bar{\ve}\oplus \bar{\eta}\ve)$ map $A$ to $D+R$ and their difference maps into the ideal $R$ of $D+R$. Therefore this difference defines, for each $t\in [0,1]$ a homomorphism $\alpha_t$ from $qA$ to $R$.
\\
We also define a homomorphism $\bar{\alpha}_t:qA\to M_2(QqA)$ by the pair of homomorphisms $\left( \bar{\eta}\ve\oplus\bar{\eta}\bar{\ve},\tau_t(\bar{\eta}\bar{\ve}\oplus\bar{\eta}\ve )\right)$ from $A$ to $M_2(Q^2A)$.
Let us denote the quotient map $QqA\to QqA/q^2A$ by $x\mapsto x^\bullet$. As already remarked above, we have $R^\bullet\cong M_2(qA)$ and we also have $(M_2(\bar{\eta} qA))^\bullet\cong M_2(qA)$. Under the quotient map $R$ becomes equal to $M_2(\bar{\eta} qA)$, $\sigma_t$ becomes equal to $\tau_t$ and therefore $\alpha_t(x)^\bullet=\bar{\alpha}_t(x)^\bullet$ for all $x\in qA$.\\
It follows that the pair $(\alpha_t,\bar{\alpha}_t)$ defines a continuous family of homomorphisms $\psi_t:q^2A \to M_2(q^2A)$. These homomorphisms are restrictions of the maps $Q^2A\to M_2 (Q^2A)$ that map $\eta\ve(x)$ and $\eta\bar{\ve}(x)$ to $ \eta\ve\oplus\bar{\eta}\bar{\ve},\,\sigma_t(\eta\bar{\ve}\oplus\bar{\eta}\ve )$ and $\bar{\eta}\ve(x),\bar{\eta}\bar{\ve}(x)$ to $\bar{\eta}\ve\oplus\bar{\eta}\bar{\ve},\,\tau_t(\bar{\eta}\bar{\ve}\oplus\bar{\eta}\ve)$, respectively.\\
For $t=0$ one easily checks  for $z\in qA$ that $\alpha_0(z)= \eta (z)\oplus\bar{\eta} (\gamma (z))$ and $\bar{\alpha}_0(z)= \bar{\eta} (z)\oplus \bar{\eta} (\gamma (z))$ where $\gamma$ denotes the restriction of the automorphism of $QA$ that interchanges $\iota$ and $\bar{\iota}$. Thus the pair $(\alpha_0 , \bar{\alpha}_0 )$ induces the homomorphism $ \id_{q^2A} \oplus 0: q^2A \to M_2(q^2 A)$.\\
For $t=1$, $\alpha_1:qA\to M_2(q^2A)$ is $\vp_A$ and $\bar{\alpha}_1$ is 0. This shows that $\psi_1 =\vp_A\pi_{qA}$.\\
It remains to show that $\pi_{qA}\vp_A$ is homotopic to $\id_{qA}\oplus 0$. The map $\pi_{qA}:q^2A\to qA$ is the restriction of the homomorphism $QqA\to qA$ that annihilates $\bar{\eta}(qA)$. Consider $\lambda_t:qA\to R\subset M_2(QqA)$ defined by the pair $\left(\eta\ve\oplus\bar{\eta}\bar{\ve},\sigma_t(\eta\bar{\ve}\oplus\bar{\eta}\ve)\right)$. We find that $\pi_{qA}\lambda_0 =\id_{qA}\oplus 0$ and $\pi_{qA}\lambda_1 =\pi_{qA}\vp_A$.
\eproof

\bremark\label{remfunct}
The map $\vp_A$ is functorial (up to stable homotopy) in the following sense:  If $\alpha:qA\to qB$ is a homomorphism between separable C*-algebras, then after stabilizing $q^2B$ the homomorphisms $q(\alpha)\vp_A$ and  $\vp_B \alpha$ are homotopic. \\
In fact, let $\sim$ denote stable homotopy equivalence. Using Proposition \ref{prop pi} to note that $\pi_{qA} \vp_A \sim \id_{qA}$ and $\vp_B \pi_{qB} \sim \id_{q^2 B}$, as well as the observation $\alpha \pi_{qA} = \pi_{qB} q(\alpha)$, we get
\[
q(\alpha) \vp_A \sim \vp_B \pi_{qB} q(\alpha) \vp_A = \vp_B \alpha \pi_{qA} \vp_A \sim \vp_B \alpha.
\]
\eremark

Given C*-algebras $X$ and $Y$ we use the standard notation $[X,Y]$ to denote the set of homotopy classes of homomorphisms from $X$ to $Y$. Thus we have $KK(X,Y) = [qX,\cK\otimes Y]$. Given $\alpha: qX\to \cK\otimes Y$ and $\beta: qY\to \cK\otimes Z$ we write $\beta\,\sharp\,\alpha$ for $(\id_\cK\otimes \beta ) \mu q(\alpha)\vp_A$, see formula \eqref{forprod}. Thus the homotopy class $[\beta\,\sharp\, \alpha]$ represents the Kasparov product of $[\alpha]$ and $[\beta]$. One way to prove the associativity of the Kasparov product consists in identifying $KK(X,Y)=[qX,\cK\otimes Y]$ with $[\cK\otimes qX,\cK\otimes qY]$ using Proposition \ref{prop pi} and to check that, under this identification the Kasparov product induced by $\sharp$ corresponds to the composition product of homomorphisms and thus is associative. This observation was stated explicitly for the first time by Skandalis in \cite{SkanKK}. We have the following proposition.

In the following we consider $qA$ as a subalgebra of $\cK \otimes qA$ as the $(1,1)$-corner embedding.
\bprop\label{prop q2}
The map $[\alpha]\mapsto [\bar{\alpha}]$ where $\bar{\alpha}=(\id_\cK\otimes \pi_B)\alpha|_{qA}$ is an isomorphism from $[\cK\otimes qA,\cK\otimes qB]$ to $[qA,\cK\otimes B]$ with inverse given by the map $[\beta]\mapsto [\beta']$ where $\beta'=\mu (\id_{\cK} \otimes q(\beta)\vp_A)$ with $\mu$ as in \eqref{forprod}.
It is multiplicative in the sense that it maps $[\beta\alpha]$ to $[\bar{\beta}\,\sharp\,\bar{\alpha}]$. In particular the product on $KK$ induced by $\sharp$ is associative.\eprop
For the proof of the proposition we need the following lemma.
\blemma\label{lemma rot}
The natural maps $q(\pi_A)$ and $\pi_{qA}$ from $q^2A$ to $qA$ are homotopic as maps to $M_2(qA)$.
\elemma
\bproof
Both homomorphisms from $q^2A$ to $qB$ are restrictions of homomorphisms from $Q^2A$ to $QB$. The first one maps $\eta\ve (x), \eta \bar{\ve} (x), \bar{\eta}\ve (x), \bar{\eta}\bar{\ve} (x)$ to $\iota (x), \bar{\iota} (x),\, 0,\, 0 $ and the second one to $\iota (x), \, 0, \bar{\iota} (x), \, 0$. The homotopy between the two is obtained by rotating in the homomorphism $q^2 A\to M_2(qA)$ which is the restriction of the homomorphism $Q^2A\to M_2(QA)$ mapping the generators to
$$\scriptsize{\left(\begin{matrix}
 \iota (x) & 0\\
 0 & 0
\end{matrix}\right)
\quad \left(\begin{matrix}
 \bar{\iota} (x) & 0\\
 0 & 0
\end{matrix}\right)\quad
\left(\begin{matrix}
  \bar{\iota} (x) & 0\\
 0 & 0
\end{matrix}\right) \quad
\left(\begin{matrix}
  \bar{\iota} (x) & 0\\
 0 & 0
\end{matrix}\right)}$$
the second and fourth term to $\scriptsize{\left(\begin{matrix}
  0 & 0\\
 0 & \bar{\iota}(x)
\end{matrix}\right)}$.
\eproof

\bproof[Proof of Proposition \ref{prop q2}]
We use $\sim$ to mean homotopic. Up to stabilisations we have
$$(\bar{\alpha})'= \mu q((\id_\cK\otimes\pi_B) \alpha|_{qA})\vp_A \stackrel{\ref{lemma rot}}{\sim} (\id_\cK\otimes\pi_{qB}) \mu q(\alpha|_{qA}) \vp_A = \pi_{\cK \otimes qB} q(\alpha|_{qA}) \vp_A  = \alpha|_{qA}  \pi_{qA}\vp_A
$$
and this is homotopic to $\alpha$ by Proposition \ref{prop pi}. Also
$$\overline{\beta'} = (\id_\cK\otimes\pi_B )\mu q(\beta)\vp_A= \beta \pi_{qA}\vp_A$$ which also is homotopic to $\beta$ by \ref{prop pi} (in both cases we have used the obvious identity $\pi_Y q(\psi) = \psi \,\pi_X: qX\to Y$ for a homomophism $\psi :X\to Y$). \\
Concerning multiplicativity we get (omitting here for clarity the stabilizations and $\mu$) for $\alpha:qA\to qB$ and $\beta: qB\to qC$ that
\bglnoz \overline{\beta\alpha}=\pi_{C}\beta\alpha \sim \pi_C\beta \alpha \pi_{qA} \vp_A \stackrel{\alpha\pi_{qA}= \pi_{qB}q(\alpha)}{=}
\pi_C\beta\pi_{qB}q(\alpha)\vp_A \\ \stackrel{\ref{lemma rot}}{\sim}\quad \pi_C\,\beta q(\pi_B)q(\alpha)\vp_A = \pi_C \beta q(\pi_B \alpha)\vp_A
= \bar{\beta}\, \sharp\, \bar{\alpha}.
\eglnoz
\eproof
\subsection{Another description of the product}
For a prequasihomomorphism $A\rightrightarrows E \rhd J$ given by the pair of homomorphisms $\alpha,\bar{\alpha}: A\to E$ we write as above $q(\alpha,\bar{\alpha})$ for the corresponding map $qA\to J$ (i.e. the restriction of $\alpha\star \bar{\alpha}$ from $QA$ to $qA$).\\
For the product of $KK$-elements $\alpha :qA\to \cK\otimes B$ and $\beta :qB\to \cK\otimes C$ only the restriction of $\beta$ to $qB_0$ matters, where $B_0$ is the hereditary subalgebra of $\cK\otimes B$, generated by the image $\alpha (qA)$,  see Remark \ref{remrest} (b). This observation leads to an alternative description of the product which we will also use to discuss associativity of the product in $KK^{nuc}$ in section \ref{s2}. In fact, for the purposes of this section it would suffice to use the smaller C*-subalgebra $\underline{B}$ of $\cK\otimes B$ generated by $\alpha (qA)$ instead of $B_0$. But we will apply the following discussion to the product in $KK^{nuc}$ in section \ref{s2} and there the choice of the hereditary subalgebra will be important.\\
With $B_0$ as above we define $\alpha_E,\bar{\alpha}_E:A \to \cM(B_0)\oplus A$ by $\alpha_E(x)= (\alpha^\circ\iota_A (x)\,,\, x)$, $\bar{\alpha}_E(x)= (\alpha^\circ\bar{\iota}_A (x)\,,\, x)$ and set $E_\alpha = C^*(B_0,\alpha_E(A),\bar{\alpha}_E(A))$. This gives an exact sequence
$0 \to B_0\to E_\alpha \stackrel{p}{\to} A \to 0$ with two splittings given by $\alpha_E,\bar{\alpha}_E : A\to E_\alpha$. Note that the prequasihomomorphism $(\alpha_E,\bar{\alpha}_E)$ represents $\alpha:qA\to B_0$ i.e. $\alpha = q(\alpha_E,\bar{\alpha}_E)$.
\blemma\label{lemext}
Let $\alpha$, $E_\alpha$ and $B_0$ be as above and $\beta:q(B_0)\to \cK\otimes C$. Let $j_E:B_0\to E_\alpha$ be the inclusion.
There is $\beta':q(E_\alpha)\to M_2(\beta(qB_0))$ such that $\beta$ is homotopic to $\beta' q(j_E)$.
\elemma
\bproof Let $\kappa_\alpha : qE_\alpha \to B_0$ be the homomorphism defined by the prequasihomomorphism $(\id_{E_\alpha},\alpha_E \circ p)$ (recall that $p:E_\alpha \to A$ is the quotient map) and set $\beta'= \beta \,\sharp \, \kappa_\alpha = \beta q(\kappa_\alpha)\vp_{E_\alpha}$. It is immediately checked that $\kappa_\alpha q(j_E)=\pi_{B_0}$ (in fact $\kappa_\alpha (\iota (x)q(y))=xy$ and $\kappa_\alpha (\bar{\iota} (x)q(y))=0$ for $x,y\in B_0$). Using the homotopy $\vp_{E_\alpha}q(j_E)\sim q^2(j_E)\vp_{B_0}$ from Remark \ref{remfunct} we get (assuming that $B$ is stable) the following homotopy
$$\beta' q(j_E) = (\beta\sharp \kappa_\alpha)q(j_E) = \beta\, q(\kappa_\alpha )\vp_{E_\alpha} q(j_E) \stackrel{\ref{remfunct}}{\sim}
\beta\, q(\kappa_\alpha )q^2(j_E)\vp_{B_0}=\beta q(\pi_{B_0}) \vp_{B_0} \stackrel{\ref{ass}}{\sim} \,\beta
$$
\eproof
Given a homomorphism $\mu : qA\to \cK\otimes B$, we denote by $\breve{\mu}$ the composition $\mu\delta$ of $\mu$ with the symmetry $\delta$ of $qA$ that exchanges the two copies of $A$. Then $\breve{\mu}$ is an additive homotopy inverse to $\mu$, i.e. we have $\mu\oplus \breve{\mu}\sim 0$ (we can rotate $\iota (x)\oplus \bar{\iota}(x)$ to $\bar{\iota} (x)\oplus\iota (x)$ in $2\times 2$-matrices).\\
Note that, if $\nu$ is a second additive homotopy inverse to $\mu$, then $\nu$ is homotopic to $\breve{\mu}$ in matrices (because $\nu \sim \nu\oplus\mu\oplus \breve{\mu}\sim 0\oplus 0\oplus \breve{\mu}$).
\bprop\label{propomega}
Let $\alpha,\beta,E_\alpha,B_0$ be as above and assume that $\beta':qE_\alpha \to \cK\otimes C$ extends $\beta$ up to homotopy as in \ref{lemext}.
If we let $C_0$ denote the hereditary subalgebra of $\cK\otimes C$ generated by $\beta(qE_\alpha)$, we get two homomorphisms $\beta'_E,\bar{\beta}'_E: E_\alpha \to E_{\beta'}$ which we can compose with $\alpha_E,\bar{\alpha}_E:A\to E_\alpha$.\\
 The homomorphism $\beta q(\alpha):q^2A\to C_0\subset \cK\otimes C$ is homotopic to $\omega q(\pi_A)$ where $\omega :qA\to C_0\subset\cK\otimes C$ is given by $\omega =q(\beta_E'\alpha_E\oplus\bar{\beta}_E'\bar{\alpha}_E \, ,\, \bar{\beta}_E'\alpha_E\oplus\beta_E\bar{\alpha}_E )$.\\
\eprop
\bproof
The homomorphism $\alpha = q(\alpha_E,\bar{\alpha}_E):qA\to B_0$ extends to the homomorphism $\alpha_E\star\bar{\alpha}_E$ from $QA$ to $E_\alpha$. As a homomorphism to $M_2(E_\alpha)$ this extended map is homotopic to $(\alpha_E\oplus 0)\star (0\oplus\bar{\alpha}_E)$. The restriction of the latter map to $qA$, which we denote by $\alpha^\oplus$, is described by $\alpha^\oplus =\alpha_E\pi_A\oplus \bar{\alpha}_E \breve{\pi}_A$. We have
$$\beta q(\alpha) \sim \beta'q(\alpha) \sim \beta' q(\alpha^\oplus)\sim \beta' q(\alpha_E \pi_A)\oplus \beta' q(\bar{\alpha}_E \breve{\pi}_A)
$$
where we have used that $\beta'$ composed with a direct sum is in $2\times 2$-matrices homotopic to the direct sum of the two compositions. By the uniqueness of the additive homotopy inverse we have that $\beta' q(\bar{\alpha}_E\breve{\pi}_A)\sim \breve{\beta}' q(\bar{\alpha}_E\pi_A)$. The result follows since $\beta' = q(\beta'_E , \bar \beta'_E)$.
\eproof
\bcor\label{coromega}
Let $\alpha,\beta,E_\alpha,B_0$ be as above and assume that $\beta$ extends up to homotopy to $\beta':qE_\alpha \to \cK\otimes C$. Then the $KK$-product $\beta\,\sharp\,\alpha$ is represented by the homomorphism $\omega:qA\to M_2(C_0)\subset \cK\otimes C$ given by $$\omega =q(\beta_E'\alpha_E\oplus\bar{\beta}_E'\bar{\alpha}_E \, ,\, \bar{\beta}_E'\alpha_E\oplus\beta'_E\bar{\alpha}_E ).$$
\ecor
\bproof
By Proposition \ref{propomega}, Proposition \ref{prop pi} and Lemma \ref{lemma rot} we have $$\beta\sharp\alpha\stackrel{\ref{remrest}}{\sim} \beta q(\alpha)\vp_A\stackrel{\ref{propomega}}{\sim} \omega q(\pi_A)\vp_A \stackrel{\ref{prop pi}}{\sim} \omega .$$
\eproof
Note that, for the formula for $\beta\sharp\alpha$ in Corollary \ref{coromega} we don't need the universal map $\vp_A$ in full but only the product $\beta \,\sharp \, \kappa_\alpha$. One could base an alternative construction of the product in $KK$ by reducing it to the special case of the product by $\kappa_\alpha$.
\subsection{Another proof for associativity}\label{sua} We follow here the discussion in Section 4 of \cite{CuGen}. Assume that we have elements in $KK(A,B)$, $KK(B,C)$, $KK(C,D)$ represented by homomorphisms $\alpha: qA\to \cK\otimes B$, $\beta: qB \to \cK\otimes C$, $\gamma: qC\to \cK\otimes D$. We define successively first $E_\alpha\supset B_0$ and $\alpha_E,\bar{\alpha}_E:A\to E_\alpha$ as above, then $\beta':qE_\alpha \to \cK\otimes C$ such that the restriction of $\beta'$ to $qB_0$ is homotopic to $\beta$. We let $C_0$ denote the hereditary subalgebra of $\cK\otimes C$ generated by $\beta'(qE_\alpha)$. Then we define $E_{\beta'}$ as before and get homomorphisms $\beta'_E,\bar{\beta}'_E:E_\alpha\to E_{\beta'}$. We then take $\gamma':qE_{\beta'} \to \cK\otimes D$ such that its restriction to $qC_0$ is homotopic to $\gamma$ and get homomorphisms $\gamma'_E,\bar{\gamma}'_E:E_{\beta'}\to E_{\gamma'}$.

We can now apply Proposition \ref{propomega} to determine the two products $\gamma'\,\sharp\, (\beta' \sharp \alpha)$ and $(\gamma'\sharp\beta' )\, \sharp\, \alpha$. They will be homotopic to $\gamma\sharp(\beta\sharp\alpha)$ and $(\gamma\sharp\beta)\sharp\alpha$. By Remark \ref{remrest} and Corollary \ref{coromega} the previous products can be described as $\gamma'\sharp \omega_1$ and $\omega_2\sharp \alpha$ with
\bglnoz
\omega_1
 =q(\beta_E'\alpha_E\oplus\bar{\beta}_E'\bar{\alpha}_E \, ,\, \bar{\beta}_E'\alpha_E\oplus\beta'_E\bar{\alpha}_E )\\
 \omega_2 =q(\gamma_E'\beta'_E\oplus\bar{\gamma}_E'\bar{\beta}'_E \, ,\, \bar{\gamma}_E'\beta'_E\oplus\gamma_E\bar{\beta}'_E )
\eglnoz
We can now apply Proposition \ref{propomega} to both products. By the special form of $\omega_1$, the homomorphisms $\gamma'_E,\bar{\gamma}_E'$ can be composed with the homomomorphisms occuring in the two components of $\omega_1$. Therefore $\gamma'$ extends to $E_{\omega_1}$ and we are in the situation of \ref{propomega}. Second, the two homomorphisms defining $\omega_2$ can be composed with $\alpha_E,\bar{\alpha}_E$ and therefore $\omega_2$ extends to $E_\alpha$. When we apply Proposition \ref{propomega} to $\gamma'\,\sharp\, (\beta' \sharp \alpha)$ and $(\gamma'\sharp\beta' )\, \sharp\, \alpha$ and use the special form of $\omega_1,\omega_2$ we find that in both cases the triple product is given by
$$q(\gamma'_E\beta'_E\alpha_E\oplus\bar{\gamma}'_E\bar{\beta}'_E\alpha_E\oplus
\gamma'_E\bar{\beta}'_E\bar{\alpha}_E\oplus\bar{\gamma}'_E\beta'_E\bar{\alpha}_E \,,\,
\bar{\gamma}'_E\beta'_E\alpha_E\oplus\gamma'_E\bar{\beta}'_E\alpha_E\oplus
\bar{\gamma}'_E\bar{\beta}'_E \bar{\alpha}_E\oplus\gamma'_E\beta'_E\bar{\alpha}_E)
$$
\section{The ideal related case}\label{id}
All ideals in C*-algebras in this section will be closed and two-sided.
\bdefin
Let $X$ be a topological space and $\cO (X)$ its lattice of open subsets. An action of $X$ on a C*-algebra $A$ with ideal lattice $\cI (A)$ is an order preserving map $ \cO (X)\ni U\mapsto A(U) \in \cI (A)$.
\edefin
Let $A,B$ be C*-algebras with an action of $X$.\nn
A homomorphism (or also a linear map) $\psi : A\to B$ is said to be $X$-equivariant if $\psi$ maps $A(U)$ to $B(U)$ for each $U\in \cO(X)$.\nn
 A homomorphism $\vp$ from $qA$ to $B$ is said to be weakly $X$-equivariant, if the maps $A\ni x\mapsto \vp (\iota (x)z), x\mapsto \vp (\bar{\iota} (x)z)$ are $X$-equivariant for each $z\in qA$.\nn
We say that $\vp: qA\to B$ is $q_X$-equivariant if the map $A\ni x\mapsto \vp(qx)$ is $X$-equivariant.\nn
Finally, given $X$ and a C*-algebra $A$ with an action of $X$, we can define actions of $X$ on $QA$ and $qA$ by letting $QA(U)$ and $qA(U)$ be the closed ideals generated by $Q(A(U))$ in $QA$ and by $Q(A(U))qA+ + qA\,Q(A(U))$ in $qA$, respectively (these are the kernels of the natural maps $QA\to Q(A/A(U))$ and $qA\to q(A/A(U))$). We denote $QA,qA$ with these actions by $Q_XA,q_XA$. Then
$$0\to q_XA \to Q_XA\to A\to 0$$ is an $X$-equivariant exact sequence with equivariant splitting $\iota: A\to Q_XA$.
\bprop Let $A,B$ be C*-algebras with an action of $X$ and $\vp$ a homomorphism $qA\to B$. The following are equivalent
\begin{itemize}
  \item $\vp$ is  weakly $X$-equivariant
  \item $\vp$ is  $q_X$-equivariant
  \item $\vp$ is  $X$-equivariant as a homomorphism $q_XA\to B$
\end{itemize}
\eprop
\bproof
Assume that $\vp$ is $q_X$-equivariant. By Proposition \ref{gen}, $qA$ is the closed span of elements $qy \, w$ for $y\in A$ and $w\in qA$. Then $\vp (\iota(x)qy\, w) = \vp (q(xy)w)-\vp (qx\,\bar{\iota}(y)w)$ is in $B(U)$ whenever $x$ is in $A(U)$ for all $y\in A$,  $w\in qA$. Similarly for $\vp (\bar \iota (x) qy \, w)$, which shows that $\vp$ is weakly $X$-equivariant.\\
Conversely, assume that $\vp$ is weakly $X$-equivariant. Let $x\in A(U)$ and $(u_\lambda )$ an approximate unit for $qA$. Then
$\vp (qx) = \lim_\lambda \vp (qx\, u_\lambda)  = \lim_\lambda\vp( (\iota (x)-\bar{\iota} (x))u_\lambda ) \in B(U)$.\\
If $\vp$ is weakly $X$-equivariant then $\vp(qA \,\iota (x)qA)$ and $\vp(qA\, \bar{\iota} (x)qA)$ are contained in $B(U)$ for all $x\in A(U)$ and thus, by definition of $q_XA(U)$ we get that $\vp (q_X A(U))\subset B(U)$.\\
Finally, if $\vp:q_XA\to B$ is $X$-equivariant, then $\vp (Q(A(U))qA)\subset B(U)$ which means that $\vp$ is weakly $X$-equivariant.
\eproof
\bdefin\label{dKKX}
Let $A,B$ be C*-algebras with an action of $X$. We define $KK(X;A,B)$ as the set of homotopy classes of weakly $X$-equivariant homomorphisms (or equivalently of $q_X$-equivariant morphisms) $qA\to \cK\otimes B$ (with homotopy in the category of such morphisms).\\
Equivalently this is the set of equivariant homotopy classes of $X$-equivariant homomorphisms $q_XA\to \cK\otimes B$.
\edefin
In the $X$-equivariant case the construction of the product actually carries over directly from section \ref{s3}. We can apply the arguments from there basically verbatim to $q_XA$ in place of $qA$ because all the maps and homotopies occuring in the discussion are naturally $X$-equivariant. In particular, the automorphism $\sigma$ used in the construction of $\vp_A$ is inner and therefore respects ideals and is $X$-equivariant. This in turn implies that $\vp_A$ also is $X$-equivariant as a map from $q_XA$ to $M_2(q_X^2A)$ with $q_X^2A=q_X(q_XA)$. Moreover, the homotopies used in the proofs of Propositions \ref{prop pi} and \ref{prop q2} are manifestly $X$-equivariant. We obtain
\bprop
Let $A$, $B$, $C$ be C*-algebras with an action of the topological space $X$. There is a natural bilinear and associative product $KK(X;A,B)\times KK(X;B,C)\to KK(X;A,C)$ which extends the composition product of $X$-equivariant homomorphisms.
\eprop

\section{$KK^{nuc}$ via the $qA$ formalism}\label{s2}
We start with a discussion of nuclear and weakly nuclear linear maps between C*-algebras. While nuclearity is most often studied in the context of completely positive maps, Pisier considered the case for more general linear maps in \cite[Chapter 12]{Pis}. Since we think that these notions have some independent interest we do this in more detail than what is actually needed for our purposes.
\bdefin\label{dnuc} Let $\rho \colon A \to B$ be a linear map between C*-algebras. We let $\|\rho\|_{\nuc}$ (the nuclear norm) denote the infimum over all $K\geq 0$ for which
\[
\rho\otimes \id \colon A\otimes _{\alg}D \to  B\otimes _{\max}D
\]
 is bounded  by $K$ for all C*-algebras $D$, if we equip $A\otimes _{\alg}D$ with  the minimal $C^\ast$-tensor norm. We say that $\rho$ is \emph{nuclear} if $\|\rho\|_{\nuc}$ is finite.
\edefin

In comparison, a linear map $\phi \colon A \to B$ between $C^\ast$-algebras is \emph{completely bounded} (resp.~\emph{weakly decomposable}\footnote{This name is motivated by the result from \cite[Chapter 14]{Pis} (which is due to Kirchberg) where this definition is shown to be equivalent to the map $\phi \colon A \to B \subseteq B^{\ast \ast}$ being decomposable, i.e.~a linear combination of completely positive maps.}) if there is a constant $K$ such that the map $\phi \otimes \id \colon A \otimes_{\alg} D \to B \otimes_{\alg} D$ is bounded in norm by $K$ when both tensor products are equipped with the minimal (resp.~maximal) $C^\ast$-tensor product.

Since it suffices to check complete boundedness for $D$ being matrix algebras, it follows that weakly decomposable maps are completely bounded.

Note that if $\rho \colon A\to B$ is nuclear (or weakly decomposable) and $\rho$ takes values in a $C^\ast$-subalgebra $B_0 \subseteq B$, the corestriction $\rho|^{B_0}$ is not necessarily nuclear (or weakly decomposable) since the map $B_0 \otimes_{\max} D \to B \otimes_{\max} D$ is not necessarily faithful. However, the map $B_0 \otimes_{\max} D \to B \otimes_{\max} D$ is faithful if $B_0 $ is a hereditary $C^\ast$-algebra so in that case $\rho|^{B_0}$ is still nuclear (or weakly decomposable). This explains why we often consider hereditary $C^\ast$-subalgebras, instead of just ordinary subalgebras, in the theory below.

If $E$ is a $C^\ast$-algebra with closed ideal $B$, a linear map $\psi \colon A\to E$ is called \emph{weakly nuclear} (relative to $B$) if $\psi b \colon A \to B$ (i.e.~the map $x \mapsto \psi(x)b$) is nuclear for all $b\in B$. We address in Remark \ref{r:weaklynucclassic} why this notion agrees with the more traditional notion of weak nuclearity.

Here are some easy observations on nuclear linear maps. If $X$ is a C*-subalgebra of a C*-algebra $Y$, we denote in the following by $\overline{X}^Y$ the hereditary subalgebra $\overline{XYX}$ of $Y$ generated by $X$.

\begin{lemma}\label{l:nuc}
Let $A,B,C,D$ be $C^\ast$-algebras.
\begin{enumerate}
\item\label{l:nuc-pointnorm} For a fixed $K\geq 0$, the set of linear maps $\rho \colon A \to B$ with $\| \rho\|_\nuc \leq K$ is closed in the point-norm topology.

\item\label{l:nuc-Banach} The set of nuclear linear maps $A\to B$ is a Banach space with respect to the nuclear norm.

\item\label{l:nuc-tensor} If $\rho \colon A \to B$ is nuclear and $D$ is a nuclear $C^\ast$-algebra, then $\id_D \otimes \rho$ extends canonically to a nuclear map $D \otimes A \to D \otimes B$.

\item\label{l:nuc-cp} If $\rho \colon A \to B$ is completely positive and nuclear then $\| \rho \|_{\nuc} = \| \rho\|$.

\item\label{l:nuc-comp} If $\phi \colon A\to B$, $\rho \colon B \to C$ and $\psi \colon C \to D$ are linear maps such that $\phi$ is completely bounded, $\rho$ is nuclear, and $\psi$ is weakly decomposable, then $\psi \rho \phi$ is nuclear.

\item\label{l:nuc-psib} If $\psi \colon A \to E$ is a homomorphism with an ideal $B\lhd E$, and if $b\in B$ such that $\psi b$ is nuclear, then $\| \psi b \|_{\nuc} \leq \| b\|$.

\item \label{l:nuc-ideal} If $\psi \colon A \to E$ is a homomorphism with an ideal $B\lhd E$, and if $X\subseteq B$ is a subset such that $B$ is generated as a closed right ideal by $X$, then $\psi$ is weakly nuclear relative to $B$ provided $\psi b$ is nuclear for all $b\in X$.
\end{enumerate}
\end{lemma}
\begin{proof}
\eqref{l:nuc-pointnorm}, \eqref{l:nuc-Banach}, and \eqref{l:nuc-comp} are immediate to verify, while \eqref{l:nuc-cp} is classical, see for instance \cite[Theorem 3.5.3]{BO-book}.

\eqref{l:nuc-tensor}: That $\id_D \otimes \rho$ extends is immediate from the definition of nuclearity of $\rho$, and nuclearity of $\id_D \otimes \rho$ follows since $\id_E \otimes \id_D \otimes \rho$ extends to a linear map
\[
E \otimes_{\min} (D\otimes A) = (E\otimes D)\otimes_{\min} A \to (E\otimes D) \otimes_{\max} B = E \otimes_{\max} (D\otimes B)
\]
bounded by $\| \rho\|_{\nuc}$ for any $C^\ast$-algebra $E$ by nuclearity of $D$ and $\rho$.

\eqref{l:nuc-psib}: Note that $\theta \colon A \to B$ given by $\theta(x) = b^\ast \psi(x) b$ is both completely positive and nuclear (it is the nuclear map $\psi b$ multiplied by $b^\ast$), and thus $\| \theta\|_{\nuc} \leq \|b\|^2$ by \eqref{l:nuc-cp}. Let $D$ be a non-zero $C^\ast$-algebra and $x = \sum_{j=1}^N a_j \otimes d_j \in A\otimes_{\alg} D$ with minimal tensor norm $\| x\|_{\min} = 1$. Then
\begin{eqnarray}
\| (\psi b  \otimes \id_D) (x) \|_{\max} &=& \|\sum_{j=1}^N \psi(a_j)b \otimes d_j\|_{\max} \nonumber\\
&=& \| \sum_{i,j=1}^N \theta(a_i^\ast a_j) \otimes d_i^\ast d_j \|_{\max}^{1/2} \nonumber\\
&=& \| (\theta \otimes \id_D)(x^\ast x) \|_{\max}^{1/2} \nonumber\\
&\leq& \| \theta\|_\nuc^{1/2} \nonumber \\
&\leq& \| b\|. \nonumber
\end{eqnarray}

\eqref{l:nuc-ideal}: This is an easy consequence of parts \eqref{l:nuc-Banach} and \eqref{l:nuc-psib}.
\end{proof}
\bremark\label{r:weaklynucclassic}
Classically a homomorphism (or completely positive map) $\psi \colon A \to E$ being weakly nuclear relative to a closed ideal $B$ means that $b^\ast \psi b \colon A \to B$ is nuclear for all $b\in B$. We will show that this agrees with our definition above.

If $\psi b$ is nuclear then clearly so is $b^\ast \psi b$ so one implication is obvious. Conversely, suppose $c^\ast \psi c$ is nuclear for all $c\in B$, so that we should show that $\psi b$ is nuclear for all $b\in B$. Let $(e_\lambda)_\lambda$ be an approximate identity in $B$. By Lemma \ref{l:nuc}\eqref{l:nuc-pointnorm} it suffices to show that there is an upper bound on the nuclear norms of the maps $e_\lambda \psi b$. By the polarisation identity we have
\[
e_\lambda \psi b = \frac{1}{4} \sum_{j=0}^3 i^j (i^j e_\lambda + b)^\ast \psi(.) (i^j e_\lambda + b)
\]
and by Lemma \ref{l:nuc}\eqref{l:nuc-cp} we obtain
\[
\| e_\lambda \psi b \|_{\nuc} \leq \frac{1}{4} \sum_{j=0}^3 \| (i^j e_\lambda + b)^\ast \psi(.) (i^j e_\lambda + b)\| \leq (1+\| b\|)^2\| \psi\|.
\]
Hence $\psi b$ is nuclear.
\eremark
If $X$ is a C*-subalgebra of the multiplier algebra $\mathcal M(Y)$, we denote by $\overline{X}^Y$ the hereditary subalgebra $XYX$ of $Y$ generated by $X$ (note that $XYX$ is a $C^\ast$-algebra by the Cohen--Hewitt factorisation theorem).
\begin{proposition}\label{pnuc}
Let $\psi\colon qA \to B$ be a homomorphism. The following are equivalent:
\begin{itemize}
\item[(i)] The map $A \ni x \mapsto \psi(qx) \in B$ is nuclear;
\item[(ii)] The maps $A \to B$ given by $x \mapsto \psi(\iota(x)y)$ and $x \mapsto \psi(\bar \iota(x)y)$ are nuclear for all $y\in qA$;
\item[(iii)] $\psi$ is represented by a prequasihomomorphism
\[
(\psi_1,\psi_2) \colon A \rightrightarrows E \rhd J \hookrightarrow B
\]
where $\psi_1,\psi_2$ are weakly nuclear relative to $J$;
\item[(iv)] If $\psi^\circ \colon QA \to \mathcal M(\overline{\psi(qA)}^B)$ is the canonical extension of $\psi$, then $\psi^\circ \iota$ and $\psi^\circ \bar \iota$ are weakly nuclear.
\item[(v)] If $E = \psi(qA)B$ is considered as a Hilbert $B$-module, the Kasparov module
\[
\left( \psi^\circ \iota \oplus \psi^\circ \bar \iota \colon A \to \mathcal B(E \oplus E^{op}),\scriptsize{ \begin{pmatrix} 0 & 1 \\ 1 & 0 \end{pmatrix}} \right)
\]
is nuclear in the sense of Skandalis.
\end{itemize}
\end{proposition}
\begin{proof}
With $E$ as in (v), $\mathcal B(E)$ is canonically isomorphic to $\mathcal M(\overline{\psi(qA)}^B)$ and hence (iv) and (v) are equivalent by \cite[1.5]{SkandKKnuc}.

(iv) implies (iii) is immediate since $\psi$ is induced by
\[
(\psi^\circ \iota, \psi^\circ \bar \iota ) \colon A \rightrightarrows \mathcal M(\overline{\psi(qA)}^B) \rhd \overline{\psi(qA)}^B \hookrightarrow B.
\]
For (iii) $\Rightarrow$ (ii) we have $x\mapsto \psi(\iota(x)y) = \psi_1(x)\psi(y)$ is nuclear for all $y\in qA$, and similarly $x\mapsto \psi(\bar \iota(x)y)$ is nuclear.

For (ii) $\Rightarrow$ (i), let for $y\in qA$ $\psi_y, \bar \psi_y\colon A \to B$ be the completely positive maps given by $\psi_y(x) = \psi(y^\ast \iota(x) y)$ and $\bar \psi_y(x) = \psi(y^\ast \bar \iota(x) y)$ which are nuclear by (ii). As these maps are completely positive, their nuclear norm $\| \psi_y\|_{\nuc} = \| \psi_y \| \leq \|y\|^2$ (Lemma \ref{l:nuc}\eqref{l:nuc-cp}), and similarly $\|\bar \psi_y\|_\nuc \leq \|y\|^2$. Hence
\[
 x \mapsto \psi(y^\ast \, qx \, y) = \psi_y(x) - \bar \psi_y(x)
\]
has nuclear norm bounded by $2\| y\|^2$. Letting $y$ range through an approximate identity for $qA$, these nuclear maps converge point-norm to $x\mapsto \psi(qx)$ and have nuclear norm bounded by $2$, so $\| x \mapsto \psi(qx)\|_\nuc \leq 2$ by Lemma \ref{l:nuc}\eqref{l:nuc-pointnorm}.

(i) $\Rightarrow$ (iv):  By  Proposition \ref{gen}, $\overline{\psi(qA)}^B$ is generated as a closed left ideal by $\{ \psi(qa): a\in A\}$. So to check that $\psi^\circ \iota$ is weakly nuclear it suffices by Lemma \ref{l:nuc}\eqref{l:nuc-ideal} to check that
\[
x \mapsto \psi^\circ \iota(x) \, \psi(qa) = \psi(\iota(x) qa) \stackrel{2.1}{=} \psi(q(xa)) - \psi(q(x)) \psi^\circ \bar \iota(a)
\]
is nuclear, which holds by Lemma \ref{l:nuc}\eqref{l:nuc-comp} (applied to the weakly decomposable maps given by right multiplication by a fixed element). Similarly $\psi^\circ \bar \iota$ is weakly nuclear.
\end{proof}

\begin{definition}
We say that a homomorphism $\psi \colon qA \to B$ is \emph{$q$-nuclear} if it satisfies the equivalent conditions in the above proposition.
\end{definition}

\bdefin
We define $KK^{nuc}(A,B)$ as the abelian group $[qA,\cK\otimes B]_\nuc$ of homotopy  classes (in the same category of maps) of $q$-nuclear homomorphisms $qA\to \cK\otimes B$.
\edefin
\bremark
The definition of $KK^{nuc}(A,B)$ from \cite{SkandKKnuc} for $A$ separable and $B$ $\sigma$-unital uses the original definition of Kasparov but assuming all Kasparov modules and homotopies are nuclear. The argument from \cite{CuGen} combined with Proposition \ref{pnuc} shows that the obvious map from Skandalis' $KK^{nuc}$-group to $[qA, \cK \otimes B]_\nuc$ is an isomorphism. This map, in particular, takes a Kasparov module induced by a prequasihomomorphism as in Proposition \ref{pnuc}(iii) (with $\cK \otimes B$ instead of $B$) to the induced $q$-nuclear homomorphism $\phi \colon qA \to \cK \otimes B$.
\eremark
\bremark
A C*-algebra $A$ is $K$-nuclear in the sense of Skandalis, if and only if the natural projection $\pi_A: qA\to A$ composed with the inclusion $A\to \cK\otimes A$ is homotopic to a $q$-nuclear homomorphism $qA\to \cK\otimes A$.
\eremark
We now discuss the product of elements in $KK^{nuc}$ by elements in $KK$. We want to see that our formula  in Subsection \ref{sub prod} for the product of two $KK$-elements represented by $\rho:qA\to \cK\otimes B$ and $\psi :qB\to \cK\otimes C$  gives a well defined element in $KK^{nuc}(A,C)$ if $\rho$ or $\psi$ is $q$-nuclear. The product, as we defined it, depends only on the restriction of $\psi$ to $q(\rho (qA))$. But if $\rho:qA\to B$ is $q$-nuclear then we don't know if $\rho:qA\to \rho (qA)$ is too. Therefore we apply the formula for the product from Section \ref{s3} to the corestrictions/restrictions $\rho_0: qA\to B_0$ and $\psi_0: qB_0\to C_0$ of $\rho$ and $\psi$, where $B_0=\overline{\rho (qA)}^B$, and $C_0 =\overline{\psi (qB_0)}^C$ are the hereditary subalgebras generated by $\rho (qA)$ and $\psi (qB_0)$. Then $\rho_0$ is $q$-nuclear iff $\rho$ is and $\rho=j_{B_0} \circ \rho_0$ for the embedding $j_{B_0}: B_0\to \cK\otimes B$ (and the same for $\psi$ and $\psi_0$). Similarly we denote by $(\psi_0\sharp \rho_0)_0$ the corestriction of $\psi_0\sharp \rho_0$ to the hereditary subalgebra $C_0$ generated by the image of $\psi_0\sharp \rho_0$. The product in $KK$ without nuclearity condition of $\psi$ and $\rho$ will be the same as the product $(\psi_0\sharp\rho_0)_0$ composed with the embedding $j_{C_0}: C_0 \hookrightarrow \cK\otimes C$ (see Remark \ref{remrest} (b)). We call $\rho_0, \psi_0$ the completed form of $\rho,\psi$ and $(\psi_0\sharp \rho_0)_0$ the completed product.\\
We consider the two maps $\eta^\psi,\bar{\eta}^\psi: B_0\to \cM(C_0)$ given by $\eta^\psi =\psi_0^\circ \iota_{B_0},\,\bar{\eta}^\psi=\psi_0^\circ \bar{\iota}_{B_0}$ (with $\iota_{B_0},\bar{\iota}_{B_0}:B_0\to QB_0$ the natural inclusions) and set $R^\psi_1=\eta^\psi (B_0)$, $R^\psi_2 =\bar{\eta}^\psi (B_0)$ and let $R^\psi$ be the C*-algebra generated in $M_2(\cM(C_0))$ by the matrices in
$$\left(\begin{matrix}
 R^\psi_1 & R^\psi_1R^\psi_2\\
 R^\psi_2R^\psi_1 & R^\psi_2
\end{matrix}\right)$$
We also denote by $J_0$ the intersection of $R^\psi$ with $M_2(C_0)$.\\
We can extend $\eta^\psi,\bar{\eta}^\psi$ to maps from the multipliers of $B_0$ to the multipliers of $R^\psi_1,R^\psi_2$ respectively. By composing these extended maps with the natural maps $\ve^\rho,\bar{\ve}^\rho:A\to \cM(B_0)$ (given by $\rho_0^\circ\iota$ and $\rho_0^\circ\bar{\iota}$) we obtain maps $\eta^\psi\ve^\rho,\eta^\psi\bar{\ve}^\rho:A\to \cM (R^\psi_1)$ and $\bar{\eta}^\psi\ve^\rho,\bar{\eta}^\psi\bar{\ve}^\rho:A\to \cM (R^\psi_2)$.\\
This means that the maps
$$
h_1^{\psi\rho} = \left(\begin{matrix}
 \eta^\psi\ve^\rho &0\\
 0 & \bar{\eta}^\psi\bar{\ve}^\rho
\end{matrix}\right) \qquad
h_2^{\psi\rho} = \left(\begin{matrix}
 \eta^\psi\bar{\ve}^\rho &0\\
 0 & \bar{\eta}^\psi\ve^\rho
\end{matrix}\right)
$$
are homomorphisms from $A$ to the multipliers of $R^\psi$.

\blemma\label{lem nuc}
If $\rho$ or $\psi$ is $q$-nuclear, then $h_1^{\psi\rho}$ and  $h_2^{\psi\rho}$ are weakly nuclear relative to $J_0$.
\elemma
\bproof
Assume that $\rho$ is weakly nuclear. Then the map $A\ni x\mapsto v\ve^\rho(x)v^*$ is nuclear for each $v\in B_0$ and the same for $\bar{\ve}^\rho$. If we apply $\eta^\psi$ to this map we see that $A\ni x\mapsto w\eta^\psi\ve^\rho(x)w^*$ is nuclear for each $w\in \eta^\psi(B_0)$.  If we multiply $w$ in this map by $y\in C_0$ on the left we find that $A\ni x\mapsto yw\eta^\psi\ve^\rho(x)w^*y^*$ is nuclear for each $w\in \eta^\psi(B_0)$ and $y\in C_0$ and the same for $\bar{\eta}^\psi$ and $\bar{\ve}^\rho$ in place of $\eta^\psi$ and/or $\ve^\rho$. By matrix multiplication this shows that the maps $A\ni x \mapsto zh_i^{\psi\rho}z^*$ are nuclear for $i=1,2$ and each $z\in J_0$.\\
Assume now that $\psi$ is $q$-nuclear.\\
If $(u_\lambda)$ is an approximate unit for $B_0$, then, by the special definition of $R^{\psi}$, we have that $zh^{\psi\rho}_1(u_\lambda)$ and $zh^{\psi\rho}_2(u_\lambda)$ tend to $z$ for each $z\in R^{\psi}$.\\
By $q$-nuclearity of $\psi$, for each $z\in J_0$ the map $A\ni x\mapsto z \eta^{\psi}(u_\lambda \ve^\rho (x) u_\lambda ^*)z^*$ is nuclear for each $\lambda$ and the same for $\bar{\eta}^{\psi}$ and $\bar{\ve}^\rho$. In the limit over $\lambda$ we get that the map $A\ni x\mapsto z \eta^\psi\ve^\rho(x)z^*$ is nuclear as well (as the set of nuclear c.p.~maps is point-norm closed) as the corresponding maps with $\eta^\psi$ and $\ve^\rho$ replaced with $\bar{\eta}^\psi$ and/or $\bar{\ve}^\rho$. This shows that for $i=1,2$ and $y\in J_0$ the maps $A\ni x \mapsto yh_i^{\psi\rho} (x)y^*$ are nuclear and thus that $h^{\psi\rho}_1,h^{\psi\rho}_2$ are weakly nuclear relative to $J_0$.
\eproof
We now examine the product of the bivariant elements represented by $\rho_0$ and $\psi_0$. As in the universal case we have that $R^\psi/J_0\cong M_2(B_0)$ and we can lift the multiplier $\left(\scriptsize{\begin{matrix}
 0 & 1\\
 1 & 0
\end{matrix}}\right)$ to a multiplier $S_0$ of $J_0$ that commutes mod $J_0$ with $\eta\ve (x)\oplus\bar{\eta}\ve (x)$ for $x \in A$. We set $F_0=e^{\frac{\pi i}{2} S_0}$ and $\sigma^\psi_t=\Ad e^{\frac{\pi i}{2} S_0 }$ and $\sigma^\psi = \sigma^\psi_1$. If $h_2^{\psi\rho}$ is weakly nuclear relative to $J_0$, so is the composition $\sigma^\psi h_2^{\psi\rho}$. The homomorphism $(\psi_0\,\sharp\,\rho_0)_0 = q(h_1^{\psi\rho},\sigma^\psi h_2^{\psi\rho}):qA\to M_2(C_0)$ represents the product and
defines an element of $KK(A,C_0)$ which, by Lemma \ref{lem nuc}, is $q$-nuclear whenever $\rho$ or $\psi$ is. We get
\bprop\label{prop nuc}
The pairing $(\psi_0,\rho_0)\mapsto j_{C_0}(\psi_0\,\sharp\,\rho_0)_0$ induces well defined bilinear products $KK^{nuc}(A,B)\times KK(B,C)\to KK^{nuc}(A,C)$ and\\ $KK(A,B)\times KK^{nuc}(B,C)\to KK^{nuc}(A,C)$.
\eprop
\bproof
The product $j_{C_0}\circ \psi_0\,\sharp\,\rho_0$ represents an element of $KK^{nuc}(A,C)$ by Lemma \ref{lem nuc} and the discussion after the lemma. It is well defined since $q$-nuclear homotopies on the side of $[qA,\cK\otimes B_0]_\nuc$ or $[qB_0,\cK\otimes C_0]_\nuc$ induce elements of $KK^{nuc}(A,B_0[0,1])$ or $KK^{nuc}(B_0,C_0[0,1])$. The product with such an element gives $q$-nuclear homotopies of the product.
\eproof
\subsection{Associativity}
Assume that we have elements in $KK(A,B)$, $KK(B,C)$, $KK(C,D)$ represented by homomorphisms $\alpha: qA\to \cK\otimes B$, $\beta: qB \to \cK\otimes C$, $\gamma: qC\to \cK\otimes D$ and assume that one of those is $q$-nuclear. In order to show that the two different products $\gamma\,\sharp\, (\beta \sharp \alpha)$ and $(\gamma\sharp\beta )\, \sharp\, \alpha$ are homotopic via a $q$-nuclear homotopy and are themselves both $q$-nuclear we can proceed exactly as in subsection \ref{sua}. Using the notation from there we obtain modified homomorphisms $\alpha, \beta', \gamma'$. By Proposition \ref{prop nuc}, $\beta',\gamma'$ will be $q$-nuclear if $\beta$ resp. $\gamma$ is. According to subsection \ref{sua} the product is given for both choices of parentheses by the homomorphism $qA\to D_0\subset \cK\otimes D$ given by
$$q(\gamma'_E\beta'_E\alpha_E\oplus\bar{\gamma}'_E\bar{\beta}'_E\alpha_E\oplus
\gamma'_E\bar{\beta}'_E\bar{\alpha}_E\oplus\bar{\gamma}'_E\beta'_E\bar{\alpha}_E \,,\,
\bar{\gamma}'_E\beta'_E\alpha_E\oplus\gamma'_E\bar{\beta}'_E\alpha_E\oplus
\bar{\gamma}'_E\bar{\beta}'_E \bar{\alpha}_E\oplus\gamma'_E\beta'_E\bar{\alpha}_E)
$$
It is $q$-nuclear by Proposition \ref{prop nuc}.
\bremark
(a)In the situation above it follows from Proposition \ref{prop nuc} that the two products with different choice of parentheses are $q$-nuclear, if one of the $\alpha,\beta,\gamma$ is. But if we have already established that the product is given by the long expression above and that $\beta'$ or $\gamma'$ is $q$-nuclear once $\beta$ or $\gamma$ is $q$-nuclear, then the $q$-nuclearity of the product is obvious. In fact we get the chain of ideals
$$\gamma_E'\beta_E'\alpha_E A \,\triangleright\, \gamma_E'\beta_E'B_0 \,\triangleright\, \gamma_E'C_0 \,\triangleright\, D_0
$$
and an analogous chain of ideals for each composition $\gamma'_E\beta'_E\alpha_E, \bar{\gamma}'_E\bar{\beta}'_E\alpha_E\,\ldots $. This shows that each of these compositions is weakly nuclear relative to $D_0$ as soon as one of the $\alpha,\beta,\gamma$ is $q$-nuclear.\\
(b) For the proof of associativity of the product in $KK^{nuc}$ we could also adapt the arguments from subsection \ref{ass} or from \cite{CuKK}, but the proof in subsection                        \ref{sua} is particularly well suited for the situation in $KK^{nuc}$.
\eremark

\section{The equivariant case}\label{sequi}
Let $G$ be a locally compact $\sigma$-compact group. A $G$-C*-algebra is a C*-algebra with an action of $G$ by automorphisms $\alpha_g,g\in G$ such that the map $G\ni g\mapsto \alpha_g(x)$ is continuous for each $x\in A$. We denote by $\cK=\cK_\Nz$ the algebra of compact operators on $\ell^2 \Nz$ and by $\cK_G$ the algebra $\cK(L^2 G)$ of compact operators on $L^2 G$. They are $G$-algebras with the trivial action and with the adjoint action $\Ad \lambda$ of $G$, respectively, where $\lambda \colon G \to \mathcal U(L^2G)$ is the left regular representation. We also denote by $\cK_{\Nz G}$ their tensor product with the tensor product action and will later use the fact that $\cK_{\Nz G}$ is equivariantly isomorphic to $\cK_{\Nz G}\otimes \cK_{\Nz G}$ (by Fell's absorption principle the tensor product of $\lambda$ by any unitary representation of $G$ is equivalent to a multiple of $\lambda$).

Given a $G$-C*-algebra $(A,\alpha)$ we consider the Hilbert $A$-module $L^2(G,A)$ with the natural action of $G$ given by $\lambda \alpha$ where $\lambda$ is the action by translation on $G$. The algebra of compact operators on $L^2(G,A)$ in the sense of Kasparov is isomorphic to $\cK_G\otimes A$. The induced action of $G$ on $\cK_G\otimes A$ is $\Ad \lambda\otimes \alpha$.

Since $A\mapsto QA$ is a functor, the action $\alpha$ induces actions of $G$ on $QA, qA$ and on $Q^2A, q^2A, R, J$ (see Section \ref{s3}) which we still denote by $\alpha$.

\bdefin\label{d1}
Given $G$-C*-algebras $(A,\alpha)$ and $(B,\beta)$ where $A$ is separable, define $KK^G(A,B)$ as the set of homotopy classes (in the category of equivariant homomorphisms) of equivariant *-homomorphisms from $\cK_{\Nz G}\otimes q(\cK_{\Nz G}\otimes A)$ to $\cK_{\Nz G}\otimes B$.
\edefin
\bremark
(a) The pair of homomorphisms $(\id\otimes \iota, id\otimes \bar{\iota})$ gives an equivariant homomorphism from $q(\cK_{\Nz G}\otimes A)$ to $\cK_{\Nz G}\otimes qA$. Therefore every equivariant homomorphism $qA \to \cK_{\Nz G}\otimes B$ (or equivalently every equivariant prequasihomomorphism $A\to \cK_{\Nz G}\otimes B$) induces by stabilization an element of $KK^G(A,B)$.\\
(b) It is a consequence of Definition \ref{d1} that the so defined $KK^G$ is the universal functor satisfying the usual properties of homotopy invariance, stability and split exactness, see Section \ref{suni}. Using the characterization of $KK^G$ by these properties in \cite{Thoms} our $KK^G$ is the same as the one of Kasparov \cite{KasInv}. Ralf Meyer has shown in \cite{MeyEqui} by direct comparison that Definition \ref{d1} gives the same functor as the one of \cite{KasInv}.\\
(c)  Using Meyer's result our construction of the product below gives an alternative definition of the product in Kasparov's $KK^G$.
\eremark
In order to describe the composition product for $KK^G$ we will use an equivariant version of the map $\vp_A$ in Section \ref{s3} this time from $q(\cK_{\Nz G}\otimes A)$ to $M_2(q^2(\cK_{\Nz G}\otimes A))$. As a first step we are now going to construct an equivariant map $\vp_0$ from $q(\cK_G\otimes A)$ to $M_2(\cK_G\otimes q^2A)$.

We consider first, as in Section 1, the algebras
$$R=\left(\begin{matrix}
 R_1 & R_1R_2\\
 R_2R_1 & R_2
\end{matrix}\right) \qquad D = C^* \left\{\left(\begin{matrix}
 \eta\ve (x) & 0\\
 0 & \bar{\eta}\ve (x)
\end{matrix}\right)\quad x\in A\right\} $$
where $R_1=\eta(qA)$, $R_2=\bar{\eta}(qA)$
as well as the ideal $J=R\cap M_2(q^2 A)$.

As in Section \ref{s3} we have that $(R+\!D)/J$ is isomorphic to the subalgebra of $M_2(Q(A))$ generated by $M_2(qA)$ together with the matrices
$$\left(\begin{matrix}
 \iota (x) & 0\\
 0 & \iota (x)
\end{matrix}\right) \quad x\in A.$$
Using the equivariant version of Proposition \ref{eqT} (Thomsen's noncommutative Tietze extension theorem) we can lift the multiplier $S_0= \left(\begin{matrix}
 0 & 1\\
 1 & 0
\end{matrix}\right)$ of $R/J$ to a self-adjoint multiplier $S$ of $J$ that commutes mod $J$ with $D$ and which satisfies $\alpha_g(S)-S\in J$ for all $g\in G$.

This multiplier $S$ can be extended to a $G$-invariant self-adjoint element $S'$ of $\cB(L^2(G,J))$ by setting $S'(\xi )(s)=S_s \xi (s)$ for $s\in G$ where $S_s=\alpha_s(S)=\alpha_s S\alpha_s^{-1}$ and where $\xi\in C_c(G,A)\subset L^2 (G,A)$. It is immediate that $S'$ is invariant for the action $\lambda\alpha$ of $G$ on $L^2(G,J)$. Thus $S'$ defines a $G$-invariant multiplier of $\cK_G\otimes J$.

The important point now is that moreover $S'$ commutes mod $\cK_G\otimes J$ with $D'=\cK_G\otimes D$.
In fact, for a typical rank 1 element of the form $|\,f_1\,\rangle \,\langle\, f_2\,|$ in $\cK_G$ with $f_1,f_2\in C_c(G,\Cz)$, $x\in D$ and $\xi\in C_c(G,J)\subset L^2(G,J)$ we get

\bglnoz\left(\left[ S',(|f_1\rangle \,\langle f_2|\otimes x)\right]\,\xi\right)(s) =f_1(s)\int (\overline{f_2(t)}(S_sx-xS_t))\xi (t)dt\qquad\qquad\qquad\qquad\\
 \qquad\quad= f_1(s)\int (\overline{f_2(t)}(S_sx-S_tx))\xi (t)dt - f_1(s)\int (\overline{f_2(t)}(S_tx-xS_t))\xi (t)dt
\eglnoz
where $S_t x-xS_t$, $S_s x-S_tx$ are in $J$ and continuous in $t$. In fact, $S$ was chosen, using \ref{eqT} to commute mod $J$ with $D$ and such that $S_s-S,\, S_t-S$ are in $J$ and continuous in $s,t$.

As in Section \ref{s3} we can now choose $F'=e^{\frac{\pi i}{2} S'}$.
Then $\Ad F'$ defines an automorphism $\sigma'$ of the multipliers of $\cK_G\otimes J$. Tensoring by $\id_{\cK_G}$ we extend the maps $\eta\ve, \eta\bar{\ve}, \bar{\eta}\ve, \bar{\eta}\bar{\ve}:A\to Q^2 A$ to homomorphisms from $\cK_G\otimes A$ to $\cK_G\otimes Q^2A$, still denoted by $\eta\ve, \eta\bar{\ve}, \bar{\eta}\ve, \bar{\eta}\bar{\ve}$.
Then the pair of homomorphisms
$$\left(\left(\begin{matrix}
 \eta\ve &0\\
 0 & \bar{\eta}\bar{\ve}
\end{matrix}\right),\sigma'\left(\begin{matrix}
 \bar{\eta}\ve &0\\
 0 & \eta\bar{\ve}
\end{matrix}\right)\right)$$
defines an equivariant homomorphism $\varphi_0: q(\cK_G\otimes A)$ to $\cK_G\otimes J$ (note that, by definition of $R$, both $\left(\begin{matrix}
 \eta\ve &0\\
 0 & \bar{\eta}\bar{\ve}
\end{matrix}\right)$ and $\left(\begin{matrix}
 \bar{\eta}\ve &0\\
 0 & \eta\bar{\ve}
\end{matrix}\right)$ map $\cK_{G}\otimes A$ to the multipliers of $\cK_{G}\otimes R$).

We can now stabilize the algebras involved in the definition of $\varphi_0 $ by $\cK_{\Nz G}$. Setting $A'=\cK_{\Nz G}\otimes A$ and using the fact that $\cK_{\Nz G}\otimes \cK_{\Nz G}\cong \cK_{\Nz G}$ we obtain the stabilized equivariant map
$$\vp'_A:\cK_{\Nz G}\otimes qA' \to \cK_{\Nz G}\otimes J'$$
where $J'=R'\cap q^2(A')$.
As in the non-equivariant case, the map $\varphi'_{A}$ induces the associative product $KK^G(A,B)\times KK^G(B,C)\to KK^G(A,C)$ as follows: \hspace{1mm}let elements of $KK^G(A,B)$ and of $KK^G(B,C)$ be represented by equivariant maps $$\cK_{\Nz G}\otimes q(\cK_{\Nz G}\otimes A)\stackrel{\mu}{\to} \cK_{\Nz G}\otimes B \quad\textrm{and}\quad \cK_{\Nz G}\otimes q(\cK_{\Nz G}\otimes B)\stackrel{\nu}{\to} \cK_{\Nz G}\otimes C$$ respectively.
Using the fact that $\cK_{\Nz G}\cong\cK_{\Nz G}\otimes\cK_{\Nz G}$, we get a map
$$q^2(\cK_{\Nz G}\otimes A)\cong q^2(\cK_{\Nz G}\otimes \cK_{\Nz G}\otimes A)\stackrel{\kappa}{\to} q(\cK_{\Nz G} \otimes q(\cK_{\Nz G}\otimes A)$$
and, using this, we can form the following composition
\bglnoz \cK_{\Nz G}\otimes q(\cK_{\Nz G}\otimes A)\stackrel{\vp'_A}{\lori}\cK_{\Nz G}\otimes q^2(\cK_{\Nz G}\otimes A) \stackrel{\kappa}{\to} \cK_{\Nz G}\otimes q(\cK_{\Nz G} \otimes q(\cK_{\Nz G}\otimes A))\\ \hfill \stackrel{\id\otimes q(\mu)}{\lori}\; \cK_{\Nz G}\otimes q(\cK_{\Nz G}\otimes B)\stackrel{\nu}{\to} \cK_{\Nz G}\otimes C\quad\eglnoz
which represents the product in $KK^G(A,C)$.
\subsection{Associativity} Associativity of the product in $KK^G$ follows as in Subsection \ref{ass} since all the isomorphisms and homotopies used there are manifestly $G$-equivariant once the automorphisms $\sigma_t$ are chosen to be equivariant.

\section{Universality and connection to the usual definitions}\label{suni}
We show now that the functors $KK(X;\, \cdot\,)$ and $KK^G$ that we have studied in Sections \ref{id} and \ref{sequi} are characterized - just like ordinary $KK$ - by split exactness together with homotopy invariance and stability in their respective category. It seems that $KK^{nuc}$ could also be characterized by a suitable more involved split exactness property for exact sequences with a weakly nuclear splitting. We leave that open - partly also because we think that such a characterization would be of minor interest.\\ Split exactness on equivariant, equivariantly split exact sequences does in fact follow for the functors $KK(X;\, \cdot\,)$ and $KK^G$ that we have studied in Sections \ref{id} and \ref{sequi} from the existence of the product, by the simple argument in \cite[2.1]{CuKK}.

\subsection{The case of ideal related $KK$-theory.}
Let $X$ be a topological space.
\bprop\label{psuni} $KK(X;\,\cdot\,,\,\cdot\,)$ is the universal functor from the category of separable C*-algebras with an action of $X$ to an additive category which is stable, homotopy invariant and split exact on exact sequences of algebras in the category with an $X$-equivariant homomorphism splitting.\eprop
\bproof
Given a C*-algebra $A$ with an action of $X$, consider the exact sequence
$$0\to q_XA \to Q_XA\to A\to 0$$ with the equivariant splitting $\iota: A\to Q_XA$. The usual argument showing that a free product of C*-algebras is $KK$-equivalent to the direct sum (see \cite{CuKK} Proposition 3.1) is compatible with the action of $X$, so that $Q_XA$ is equivalent in $KK(X;\cdot,\cdot)$ to $A\oplus A$ with the natural action of $X$ - just by homotopy invariance and stability. Let now $F$ be a functor from the category of separable C*-algebras with an $X$-action to an additive category which is stable, homotopy invariant and equivariantly split exact. Then $F(Q_XA)$ is isomorphic, via the natural map, to $F(A\oplus A)=F(A)\oplus F(A)$ and by split exactness consequently $F(q_XA)\cong F(A)$. By Definition \ref{dKKX} every element of $KK(X;A,B)$ is represented by an $X$-equivariant homomorphism $q_XA\to \cK\otimes B$. Applying $F$ to the homotopy class of such a homomorphism we get a morphism $F(A)\cong F(q_XA)\to F(\cK\otimes B)\cong F(B)$. Since the isomorphisms involved are natural this morphism is uniquely determined.\\
Conversely $KK(X;\cdot\,)$ is homotopy invariant, stable and splits on $X$-equi\-variant\-ly split exact sequences.
\eproof
\subsection{The case of $KK^G$.}\label{uniG}
If $G$ is a locally compact $\sigma$-compact group we also have
\bprop (cf.\cite{MeyEqui})
$KK^G$ is the universal functor on the category of separable $G$-C*-algebras which is homotopy invariant, stable under tensor product by $\cK_{\Nz G}$ and split exact on extensions $0\to I\to E\to A\to 0$ of $G$-C*-algebras with an equivariant splitting homomorphism $s:A\to E$.
\eprop
\bproof
Let $F$ be a functor with the given properties from the category of $G$-C*-algebras to an additive category and set $A'=\cK_{\Nz G}\otimes A$. Homotopy invariance and stability of $F$ imply that $F(QA')\cong F(A'\oplus A')$ (by the argument in \cite{CuKK} Proposition 3.1 which is compatible with the action of $G$). Split exactness implies that $F(QA')\cong (F(qA')\oplus F(A')$ and finally that $F(qA')\cong F(A')$ naturally. Since also $F(A')\cong F(A)$ for all $A$ by stability, the assertion then follows from the definition of $KK^G$, see \ref{d1}.\\
Conversely, $KK^G$ is equivariantly split exact by the remark at the beginning of the section.
\eproof
\subsection{Connection to the usual definitions}
The usual definitions of the different versions of $KK(A,B)$ are based on $A$-$B$ Kasparov modules $(E,F)$ with additional structure. In such a Kasparov module one can always assume that $F=F^*$ and $F^2=1$. Conjugation of the (first component for the $\Zz/2$-grading of the) left action $\vp$ of $A$ on $E$ by $F$ gives a second homomorphism $\bar{\vp}:A \to \cB(E)$. Depending on the situation, $\vp$ will `weakly' respect the additional structure ($X$-equivariance, $G$-equivariance or nuclearity respectively). Now in order to get a homomorphism from $qA$ to $\cK(E)$ respecting the additional structure we need to know that $\bar{\vp}$ also respects the structure `weakly'. Since $\bar{\vp} =\Ad F \vp$, and $\Ad F$ is inner, this is automatic for $X$-equivariance. In the case of $KK^G$ this has been established in the paper by Ralf Meyer. In the case of $KK^{nuc}$ the equivalence between $q$-nuclear homomorphisms $qA\to \cK(E)$ and nuclear Kasparov modules has been shown in Proposition \ref{pnuc}. In the case of $KK^G$ and $KK(X)$ we get the equivalence then from the universality of our definition.

\end{document}